\documentclass[12pt]{amsart}

\usepackage[nobysame]{amsrefs}

\AtBeginDocument{\def\MR#1{}}

\usepackage[UKenglish]{babel}
\usepackage[UKenglish]{datetime}

\usepackage{ae, aecompl}
\usepackage{amscd}
\usepackage{amsfonts}
\usepackage{amsmath}
\usepackage{amssymb}
\usepackage{amsthm}
\usepackage{cases}
\usepackage{chngcntr}
\usepackage{cite}
\usepackage{color}
\usepackage{graphicx}
\usepackage{latexsym}
\usepackage{bbm} 
\usepackage{bbold}
\usepackage{microtype}
\usepackage{qsymbols}
\usepackage[Symbolsmallscale]{upgreek} 
\usepackage[active]{srcltx}
\usepackage{tikz}
\usepackage{url}
\usepackage{verbatim} 

\usepackage{hyperref}
\usepackage{enumitem}
\usepackage{cleveref}

\usepackage{mathptmx}

\newcommand{\bbfont}{\mathbbm}

\newcommand{\CC}{{\bbfont C}}

\newcommand{\RR}{{\bbfont R}}

\newcommand{\ZZ}{{\bbfont Z}}

\newcommand{\upC}{{\mathrm{C}}}

\newcommand{\upL}{{\mathrm{L}}}

\newcommand{\norm}[1]{{\lVert #1 \rVert}}
\newcommand{\angles}[1]{{\langle #1\rangle}}
\newcommand{\braces}[1]{{\{ #1\}}}

\newcommand{\bigpars}[1]{{\big( #1\big)}}

\newcommand{\set}[1]{\braces{\,#1\,}}

\newcommand{\cont}{{\upC}}

\newcommand{\Calgebra}{\ensuremath{{\upC}^\ast}\!-algebra}
\newcommand{\Calgebras}{\ensuremath{{\upC}^\ast}\!-algebras}

\newcommand{\idmap}{{\mathrm{id}}}

\newcommand{\spanlc}[1][]{\operatorname{span_{#1}}} 

\theoremstyle{plain}

\newtheorem{theorem}{Theorem}[section]
\newtheorem{proposition}[theorem]{Proposition}
\newtheorem{lemma}[theorem]{Lemma}

\newtheorem*{theorem*}{Theorem}
\newtheorem*{proposition*}{Proposition}
\newtheorem*{lemma*}{Lemma}
\newtheorem*{corollary*}{Corollary}
\newtheorem*{conjecture*}{Conjecture}

\theoremstyle{definition}

\newtheorem{definition}[theorem]{Definition}

\newtheorem{remark}[theorem]{Remark}

\newtheorem*{definition*}{Definition}
\newtheorem*{example*}{Example}
\newtheorem*{remark*}{Remark}
\newtheorem*{assumption*}{Assumption}

\setlist[enumerate,1]{label=\textup{(\arabic*)},ref=\arabic*}
\setlist[enumerate,2]{label=\textup{(\alph*)},ref=\arabic{enumi}.\alph*}
\setlist[enumerate,3]{label=\textup{(\roman*)},ref=\arabic{enumi}.\alph{enumii}.\roman*}
\setlist[enumerate,4]{label=\textup{(\Alph*)},ref=\arabic{enumi}.\alph{enumii}.\roman{enumiii}.\Alph*}

\crefname{theorem}{Theorem}{Theorems}
\crefname{proposition}{Proposition}{Propositions}
\crefname{lemma}{Lemma}{Lemmas}
\crefname{corollary}{Corollary}{Corollaries}
\crefname{conjecture}{Conjecture}{Conjectures}
\crefname{definition}{Definition}{Definitions}
\crefname{example}{Example}{Examples}
\crefname{remark}{Remark}{Remarks}
\crefname{assumption}{Assumption}{Assumptions}

\crefname{equation}{equation}{equations}
\crefname{enumi}{part}{parts}
\crefname{enumii}{part}{parts}
\crefname{enumiii}{part}{parts}
\crefname{enumiv}{part}{parts}

\numberwithin{equation}{section}

\allowdisplaybreaks 

\newcommand{\Aut}{\mathrm{Aut}}

\newcommand{\dynsys}{(A,G,\alpha)}
\newcommand{\lonealgebra}{\ell^1(G,A;\alpha)}
\newcommand{\Lonealgebra}{\upL^1(G,A;\alpha)}

\newcommand{\Cstar}{\ensuremath{{\upC}^\ast}\!}

\newcommand{\laction}[2]{{\upL}_{#2}^{#1}}

\newcommand{\mean}{m}

\newcommand{\linf}[1]{\ell^\infty(#1)}

\newcommand{\fspaceletter}{\Phi}
\newcommand{\fspace}[2]{{\fspaceletter}_{#2}^{#1}}

\newcommand{\fletter}{\varphi}
\newcommand{\f}{\fletter}
\newcommand{\phifunction}[2]{\fletter_{#2}^{#1}}

\newcommand{\onefunction}[1]{{\mathbf 1}_{#1}}

\begin{document}


\title[Amenable crossed product Banach algebras] {Amenable crossed product Banach algebras associated with a class of ${\mathbf {C}}^{\mathbf{\ast}}$-dynamical systems. II}

\author{Marcel de Jeu}
\address{Mathematical Institute\\
Leiden University\\
P.O.\ Box 9512\\
2300 RA Leiden\\
the Netherlands}
\email{mdejeu@math.leidenuniv.nl}

\author{Rachid El Harti}
\address{Department of Mathematics and Computer Sciences\\
Faculty of Sciences and Techniques\\
University Hassan I, BP 577 Settat\\
Morocco}
\email{rachid.elharti@uhp.ac.ma
}

\author{Paulo R.\ Pinto}
\address{Department of Mathematics\\
Instituto Superior T\'{e}cnico\\
University of Lisbon\\
Av.\ Rovisco Pais 1\\
1049-001 Lisbon\\
Portugal}
\email{ppinto@math.tecnico.ulisboa.pt}








\begin{abstract}We prove that the crossed product Banach algebra  $\ell^1(G,A;\alpha)$ that is associated with a ${\mathrm C}^\ast$-dynamical system $(A,G,\alpha)$ is amenable if~$G$ is a discrete amenable group and~$A$ is a strongly amenable ${\mathrm C}^\ast$-algebra. This is a consequence of the combination of a more general result with Paterson's characterisation of strongly amenable unital $\mathrm{C}^\ast$-algebras in terms of invariant means for their unitary groups.
\end{abstract}

\maketitle


\section{Introduction and overview}\label{sec:introduction_and_overview}

If $\dynsys$ is a \Cstar-dynamical system, where~$A$ is a nuclear \Calgebra\ and~$G$ is an amenable locally compact Hausdorff topological group, then the crossed product \Calgebra\ $A\rtimes_\alpha G$ is a nuclear \Calgebra; see e.g.~\cite{gootman:1982},  \cite[Proposition~14]{green:1977}, or \cite[Theorem~7.18]{williams_CROSSED_PRODUCTS_OF_C-STAR-ALGEBRAS:2007}. Using Connes' work in \cite{connes:1978} and Haagerup's in \cite{haagerup:1983}, one can equivalently say that $A\rtimes_\alpha G$ is an amenable  Banach algebra if~$G$ is an amenable locally compact Hausdorff topological group and~$A$ is an amenable \Calgebra. Here a Banach algebra~$A$ is called amenable if every bounded derivation of~$A$ with values in a dual Banach $A$-bimodule is inner, and a topological group~$G$ is called amenable if there exists a left invariant state on the unital \Calgebra\ of bounded right uniformly continuous complex valued functions on~$G$. For a locally compact Hausdorff topological group~$G$, this is equivalent to the existence of left invariant states on other unital \Calgebras\ of (equivalence classes of) functions on~$G$; see e.g.\ \cite[Definition~1.1.4, Theorem~1.1.9, and Theorem~1.1.11]{runde_LECTURES_ON_AMENABILITY:2002}. If~$G$ is discrete, amenability is simply the existence of a left invariant state on $\linf{G}$. 

The \Calgebra\ $A\rtimes_\alpha G$ is the enveloping \Calgebra\ of the twisted convolution algebra $\Lonealgebra$. In view of the above reformulation of the nuclearity result in terms of amenability, it seems natural to inquire whether $\Lonealgebra$ is perhaps also amenable under the same conditions on~$A$ and~$G$. Apart from its intrinsic interest, this would also provide an alternative approach to the nuclearity of $A\rtimes_\alpha G$. Indeed, since the inclusion of the latter in the former is continuous (even contractive) with dense image, we could then use \cite[Proposition~2.3.1]{runde_LECTURES_ON_AMENABILITY:2002} to conclude that $A\rtimes_\alpha G$ is amenable, and therefore nuclear.

Not much seems to be known about the amenability of $\Lonealgebra$, or, for that matter, of other Banach algebras of $\upL^1$-type at all. There is, of course, Johnson's result  for $A=\CC$: if~$G$ is an amenable locally compact Hausdorff topological group, then $\upL^1(G)$ is amenable.  See e.g\ \cite[Theorem~2.5]{johnson_COHOMOLOGY_IN_BANACH_ALGEBRAS:1972} for this and its converse; the latter is due to Ringrose. If~$G$ is discrete, then a little  more is known. The algebra $\lonealgebra$ (we shall give its definition in Section~\ref{sec:preliminaries}) is amenable if~$A$ is a commutative or finite dimensional \Calgebra\ (see \cite[Theorem~2.4]{de_jeu_el_harti_pinto:2017b}); it is unknown whether there is a converse of some kind involving properties of~$G$.  For general~$G$ and (Banach) algebra~$A$, one can introduce a weight $\omega:G\to\RR_{\geq 0}$ and arrive at a generalised Beurling algebra $\upL^1(G,A,\omega;\alpha)$ as in \cite[Definition~5.4]{de_jeu_messerschmidt_wortel_UNPUBLISHED:2013}. For $G=\ZZ$ and $A=\CC$, it is known (see \cite[Theorem~2.4]{bade_curtis_dales:1987}) that for certain weights the ensuing Beurling algebras are amenable (or weakly amenable), whereas for others they are not. The authors are not aware of any other results on amenability for $\upL^1$-type Banach algebras.

In this paper, we show that $\lonealgebra$ is an amenable Banach algebra if~$G$ is a discrete amenable group and~$A$ is a strongly amenable not necessarily unital \Calgebra\ (see \cref{res:lonealgebra_is_amenable_if_A_is_strongly_amenable}). The latter notion has been introduced by Johnson (see \cite[p.70]{johnson_COHOMOLOGY_IN_BANACH_ALGEBRAS:1972} or \cite[p.~313]{haagerup:1983}): a unital \Calgebra\ with unitary group~$U$ is said to be strongly amenable if, for every bounded derivation~$D$ of~$A$ with values in a dual Banach~$A$-bimodule~$E^\ast$, there exists $x^\ast$ in the weak*-closed convex hull of $\set{-Du\cdot u^{-1} : u\in U}$ in~$E^\ast$ such that $Da=a\cdot x - x\cdot a$ for all $a\in A$. A non-unital \Calgebra\ is said to be strongly amenable if its unitisation is. Every strongly amenable Banach algebra is amenable, but the converse does not hold, as is shown by the Cuntz algebras ${\mathcal O}_n$ for $n\geq 2$; see~\cite{rosenberg:1977}. All Type I \Calgebras\ (equivalently: all postliminal \Calgebras; see \cite[Remark~9.5.9]{dixmier_C-STAR-ALGEBRAS_ENGLISH_NORTH_HOLLAND_EDITION:1977})  are strongly amenable; see \cite[Theorem~7.9]{johnson_COHOMOLOGY_IN_BANACH_ALGEBRAS:1972}. Thus our present result covers a reasonably wide class of examples, and, in particular,  it implies our previous result that $\lonealgebra$ is amenable if~$A$ is a commutative or finite dimensional \Calgebra.

We shall now explain the structure of the proof, which will also make clear how strong amenability of~$A$ enters the picture in a natural way, replacing the weaker requirement of amenability that was the initial Ansatz in the above discussion. Actually, it will become clear which (presumably) weaker condition than strong amenability is sufficient for $\lonealgebra$ to be amenable.

Let us sketch how one could attempt (and fail) to prove\textemdash along the lines of \cite{de_jeu_el_harti_pinto:2017b}\textemdash that $\lonealgebra$ is an amenable Banach algebra if~$G$ is a discrete amenable group and~$A$ is only known to be an amenable \Calgebra. First of all, it follows from \cite[Corollary~2.3.11]{runde_LECTURES_ON_AMENABILITY:2002} and \cite[Lemma~2.2]{de_jeu_el_harti_pinto:2017b} that it is equivalent to attempt this with~$A$ also unital, so let us assume this. In that case (we refer to \cref{res:lonealgebra_contains_semidirect_product} for details), $\lonealgebra$ contains a group~$H$ that is generated as an abstract group by~$G$ and the unitary group~$U$ of~$A$, and that has~$U$ as a normal subgroup. In fact, the group~$H$ is isomorphic to $U\rtimes_\alpha G$ as abstract groups, but we shall not need this more precise statement. The important point is that the closed linear span of~$H$ equals $\lonealgebra$. Therefore, if~$D$ is a bounded derivation of $\lonealgebra$ with values in a dual $\lonealgebra$-bimodule~$E^\ast$, and if we want to show that~$D$ is inner, we need only prove that its restriction to~$H$ is inner. This observation is already used in \cite{de_jeu_el_harti_pinto:2017b}. In~\cite{de_jeu_el_harti_pinto:2017b}, the proof then proceeds by supplying~$U$ with the inherited norm topology of~$\lonealgebra$ if~$A$ is finite dimensional, or with the discrete topology if~$A$ is commutative. Then~$U$ is an amenable locally compact Hausdorff topological group in both cases, and a known stability property for such groups (see \cite[Proposition~13.4]{pier_AMENABLE_LOCALLY_COMPACT_GROUPS:1984}) then implies that~$H$ is an amenable locally compact Hausdorff topological group as well. Consequently, \cite[Theorem~11.8.(ii)]{pier_AMENABLE_LOCALLY_COMPACT_GROUPS:1984} (see also \cite[p.~17--18 and p.~99]{pier_AMENABLE_LOCALLY_COMPACT_GROUPS:1984}) shows that~$D$ is inner on~$H$. This concludes the proof in \cite{de_jeu_el_harti_pinto:2017b}. Inspection shows that the proof of the result on innerness that is invoked (i.e.\ of \cite[Theorem~11.8.(ii)]{pier_AMENABLE_LOCALLY_COMPACT_GROUPS:1984}) is ultimately based (see \cite[proof of Lemma~11.6]{pier_AMENABLE_LOCALLY_COMPACT_GROUPS:1984}) on Johnson's archetypical argument (see \cite[p.~33]{johnson_COHOMOLOGY_IN_BANACH_ALGEBRAS:1972}) to show that, in a suitable context, bounded derivations of a group with values in dual Banach bimodules are inner if there exists an appropriate left invariant mean on the group. The same is thus true for our earlier result \cite[Theorem~2.4]{de_jeu_el_harti_pinto:2017b}: surviving all layers if~$A$ is commutative or finite dimensional, it is this argument that provides the key to the amenability of $\lonealgebra$ in \cite{de_jeu_el_harti_pinto:2017b}.

For general amenable~$A$, the natural follow-up after the observation that one need only prove that~$D$ is inner on~$H$ does not seem to work. We shall now explain this.

For a general unital \Calgebra~$A$, it has been established by Paterson (see \cite[Theorem~2]{paterson:1992}) that its unitary group~$U$ is a Hausdorff topological group in the weak topology that it inherits from the Banach space~$A$, and that~$U$ is amenable in that weak topology if~$A$ is amenable (in fact, this characterises amenable~$A$; see \cite[Theorem~2]{paterson:1992}). In view of what has worked earlier this is an encouraging starting point, since we are indeed in that situation. So let us supply~$U$ with the weak topology inherited from~$A$, and, for convenience, let us assume\textemdash this could perhaps be another matter\textemdash that we can show that $H\simeq U\rtimes_\alpha G$ is a topological group in the product topology, and that it is then amenable. In that case, one has a left $H$-invariant mean on the bounded right uniformly continuous functions on~$H$ to work with, and the next step would presumably be to use Johnson's argument to show that this implies that a bounded derivation on~$H$ with values in a dual Banach $H$-bimodule~$E^\ast$ is inner. As earlier, this would then conclude the proof. It is at this point, however, that an obstacle arises. In order to be able to apply Johnson's argument to~$H$, one needs that, for all $x\in E$, the function $h\mapsto\angles{x,Dh\cdot h^{-1}}$ is in the function space on which the left invariant mean living on~$H$ acts. In the presumed situation, it should, therefore, be bounded and right uniformly continuous. In particular, its restriction to~$U$ should be right uniformly continuous. However, there seems to be no reason why this should in general be the case or, for that matter, why it should even be continuous. If the actions of~$U$ on~$E$ are strongly continuous, then these restricted functions are easily seen to be right uniformly continuous, as required, but we have no guarantee that this is the case if the $U$-actions and the derivation~$D$ of~$U$ originate from an enveloping $\lonealgebra$-bimodule structure. The point is that we are working with the group~$U$ in its inherited weak topology from~$A$, and not in the inherited norm topology from~$A$. In the latter topology the actions of~$U$ on~$E$ are evidently strongly continuous, but there seems to be no reason why this should still be the case for the weak topology on $U$, and it is the latter topology we must work with if we want to have the amenability of~$U$ from Paterson's result at our disposal. Thus this attempt, based on combining Paterson's result for amenable unital \Calgebras\ and Johnson's argument exploiting a left invariant mean on the right uniformly continuous functions on a topological group, runs aground.

It is at this point that Paterson's characterisation of strongly amenable unital \Calgebras\ (see \cite[Theorem~2]{paterson:1991} and the left/right discussion preceding \cref{res:lonealgebra_is_amenable_if_A_is_strongly_amenable}) comes in to overcome this obstruction originating from having the `wrong' topology on~$U$. Indeed, if~$A$ is strongly amenable, then there exists a left invariant mean on a space of functions on~$U$ that makes no reference to a specific topology on~$U$ at all, but that is naturally associated with the bounded bilinear forms on~$A$. It is immediate (see the proof of \cref{res:lonealgebra_is_amenable_if_A_is_strongly_amenable}) that the functions $u\mapsto\angles{x,Du\cdot u^{-1}}$ are in this space for all $x\in E$, together with the constants. It is then not too difficult\textemdash it is here that the amenability of the discrete group~$G$ is used; see the proof of \cref{res:core_argument_for_mean})\textemdash to construct a left invariant mean on the minimal space of functions on~$H$ on which one needs such a mean in order to be able to make Johnson's argument work, i.e.\ on the space of functions on~$H$ that is spanned by the constants and the functions $h\mapsto\angles{x,Dh\cdot h^{-1}}$ for $x\in E$ (see the proof of \cref{res:core_argument_for_innerness}). With this available,~$D$ is seen to be inner on~$H$, and then we are done. It is in this way, by combining Paterson's result \cite[Theorem~2]{paterson:1991} for strongly amenable unital \Calgebras\ with a somewhat more careful inspection of what the minimal requirements are in order to be able to apply Johnson's argument, that the proof is then concluded for such algebras after all. As will become apparent from Section~\ref{sec:amenability_of_abstract_groups}, these minimal requirements are not related  to topology at all. In the end, therefore, there is no role anymore for the amenability of~$U$ in the inherited weak topology of~$A$, even though at first this seemed to be the most natural thing to start with.

Actually, as may already be obvious from the above discussion, one does not really need that~$A$ is strongly amenable for $\lonealgebra$ to be amenable. The existence of a left invariant mean on the space of functions on~$U$ that is spanned by the constants and the functions $u\mapsto\angles{x,Du\cdot u^{-1}}$ for $x\in E$ is already sufficient, and Paterson's result \cite[Theorem~2]{paterson:1991} `only' implies that this is certainly the case if~$A$ is strongly amenable. It is, therefore, perhaps more precise to regard \cref{res:lonealgebra_is_amenable_general_case}, which is still in the general context, as the most prominent result of this paper. The fact that $\lonealgebra$ is amenable if~$A$ is strongly amenable and~$G$ is amenable is an appealing special case thereof.

This paper is organised as follows.

Section~\ref{sec:preliminaries} contains the necessary terminology and definitions, including that of the Banach algebra $\lonealgebra$.

Section~\ref{sec:amenability_of_abstract_groups} is concerned with prudent hypotheses implying that an abstract group has the property that every bounded derivation with values in a dual Banach bimodule is inner.

Section~\ref{sec:amenability_of_lonealgebra} contains the main results \cref{res:lonealgebra_is_amenable_general_case,res:lonealgebra_is_amenable_if_A_is_strongly_amenable}
on the amenability of $\lonealgebra$.

In Section~\ref{sec:converses} we briefly discuss converse implications. The situation here is largely open, with only a limited number of results available that can be derived via the detour  of the enveloping \Calgebra\ $A\rtimes_\alpha G$ of $\lonealgebra$. It is also argued here that this detour could involve loss of information, so that proper $\upL^1$-type arguments are needed.

For a discussion of a surmised general framework for amenability of crossed products of Banach algebras (as in \cite[Definition~3.2]{de_jeu_messerschmidt_wortel_UNPUBLISHED:2013}) that are associated with \Cstar-dynamical systems we refer to \cite[Section~3]{de_jeu_el_harti_pinto:2017b}.

\section{Preliminaries}\label{sec:preliminaries}

We start by establishing some terminology and notation. Since the definitions in the literature can slightly vary from author to author (or even greatly in the case of `left' and `right'), we shall also define the most basic notions.

If~$G$ is an abstract group, then $\linf{G}$ denotes the space of all bounded complex valued functions on~$G$. All subspaces of $\linf{G}$ will be assumed to carry the supremum norm.  For $g_0\in G$, define the \emph{left translation} $\laction{g_0}{G}\colon\linf{G}\to\linf{G}$ by $\laction{g_0}{G}\f(g)=\f(g_0g)$ for $\f\in\linf{G}$ and $g\in G$. Then $\laction{g_1g_2}{G}=\laction{g_2}{G}\laction{g_1}{G}$ for $g_1,g_2\in G$. The definition of a \emph{right translations} is similar, without an inverse. We have included the group in the notation, as later on there will be several groups occurring simultaneously.

If $\fspace{}{G}\subset\linf{G}$ is a not necessarily closed subspace, then $\fspace{}{G}$ is said to be \emph{left $G$-invariant} if it is invariant under $\laction{g}{G}$ for all $g\in G $. An element $x_{\fspace{}{G}}^\ast$ of the norm dual $\fspace{\,\ast}{G}$ of a left $G$-invariant subspace ${\fspace{}{G}}$ of $\linf{G}$ is said to be left $G$-invariant if $\angles{\laction{g}{G}\f,x_{\fspace{}{G}}^\ast}=\angles{\f,x_{\fspace{}{G}}^\ast}$ for all $\f\in\fspace{}{G}$ and $g\in G$.  A \emph{right~$G$-invariant} subspace and a right~$G$-invariant element of its dual are similarly defined.

Although we have employed it in Section~\ref{sec:introduction_and_overview}, we shall not use the terminology of `means' in the sequel, but simply state the properties that an element~$\mean$ of the norm dual of a subspace of $\linf{G}$ is required to have. As we shall see, the state-like property that $\norm{\mean}=1=\mean(\onefunction{G})$  (here~$\onefunction{G}$ denotes the function on~$G$ that is identically~$1$) will never be needed; see, however, part~\ref{rem:existence_of_a_left_invariant_state} of  \cref{rem:no_state_like_condition_and_existence_of_a_left_invariant_state}.

If~$G$ is an abstract group, then a \emph{Banach left $G$-module} is a Banach space~$E$ that is supplied with a left $G$-action with the property that there exists $K\geq 0$ such that $\norm{g\cdot x}\leq K\norm{x}$ for all $g\in G$ an $x\in E$. We do not assume that the $G$-action is unital. The definitions of a \emph{Banach right $G$-module} and of a \emph{Banach~$G$-bimodule} are analogous. A \emph {dual Banach $G$-bimodule} is obtained from a Banach $G$-bimodule by taking the adjoint actions.

If~$G$ is an abstract group, and~$E$ is a Banach~$G$-bimodule, then a \emph{derivation of the group~$G$ with values in~$E$} is a map $D\colon G\mapsto E$ such that $D(g_1g_2)=Dg_1\cdot g_2 + g_1\cdot Dg_2$ for all $g_1,g_2\in G$. It is a \emph{bounded derivation of the group~$G$} if $D(G)$ is a norm bounded subset of~$E$. For $x\in E$, the map $g\mapsto g\cdot x - x\cdot g$ for $g\in G$ is a bounded derivation of~$G$; such a derivation is called an \emph{inner derivation of the group~$G$}

If~$A$ is a Banach algebra, then a \emph{Banach left $A$-module} is a Banach space~$E$ that is supplied with a left $A$-action with the property that there exists a constant $K\geq 0$ such that $\norm{a\cdot x}\leq K\norm{a}\norm{x}$ for all $a\in A$ and $x\in E$. We do not assume that (if applicable) the $A$-action is unital. The definitions of a \emph{Banach right $A$-module} and a \emph{Banach~$A$-bimodule} are analogous. A \emph{dual Banach $A$-bimodule} is obtained from a Banach $A$-bimodule by taking the adjoint actions.

If~$A$ is a Banach algebra, and~$E$ is a Banach~$A$-bimodule, then a \emph{derivation of the Banach algebra~$A$ with values in~$E$} is a linear map $D\colon A\mapsto E$ such that $D(a_1 a_2)=Da_1\cdot a_2 + a_1\cdot Da_2$ for all $a_1,a_2\in A$. It is a \emph{bounded derivation of the Banach algebra~$A$} if~$D$ is a bounded operator between the normed spaces~$A$ and~$E$. For $x\in E$, the map $a\mapsto a\cdot x - x\cdot a$ for $a\in A$ is a bounded derivation of~$A$. Such a derivation is called an \emph{inner derivation of the Banach algebra~$A$}. The Banach algebra~$A$ is an \emph{amenable Banach algebra} if every bounded derivation of~$A$ with values in a dual Banach~$A$-bimodule is an inner derivation.

After this part on terminology and notation, we turn to the description of the Banach algebra $\lonealgebra$.

Let $(A,G,\alpha)$ be a \Cstar-dynamical system with~$G$ discrete, i.e.\ let $\alpha\colon G\to \Aut(A)$ be an action of a discrete group~$G$ as $^\ast$-automorphisms of a (not necessarily unital) \Calgebra~$A$. Let
\begin{equation*}
\lonealgebra=\{\,\,\texttt{\bf a}: G\longrightarrow A:\ \Vert \texttt{\bf a}\Vert:=\sum_{g\in G} \Vert a_g\Vert <\infty\,\},
\end{equation*}
where for typographical reasons we have written $a_g$ for $\texttt{\bf a}(g)$.
We supply $\lonealgebra$ with the usual twisted convolution product and involution, defined by
\begin{align}
(\texttt{\bf a}\texttt{\bf a}^\prime)(g)&=\sum_{h\in G} a_h \cdot \alpha_h(a_{h^{-1}g}^\prime)\label{eq:prod}\\
\intertext{for $\texttt{\bf a},\texttt{\bf a}^\prime\in \lonealgebra)$ and $g\in G$, and by}
{\texttt{\bf a}}^\ast(g)&=\alpha_g((a_{g^{-1}})^\ast)\label{eq:involution}
\end{align}
for $\texttt{\bf a}\in \lonealgebra)$ and $g\in G$, respectively. Then $\lonealgebra$ becomes a Banach algebra with isometric involution. The usual convolution algebra $\ell^1(G,A)$ is the special case $\ell^1(G,A;\textrm{triv})$. Specialising further, if $A=\mathbb C$, then $\ell^1(G,\mathbb C;\textrm{triv})$ is the usual group algebra $\ell^1(G)$.

Suppose that~$A$ is unital. In that case, there is a more convenient model for $\lonealgebra$, as we shall now indicate.

For $g\in G$, let $\delta_g\colon G\to A$ be defined by
\[
\delta_g(h)=\left\{
\begin{array}{ll}
1_A & \hbox{if}\ h=g;\\
0_A & \hbox{if}\ h\not=g,
\end{array} \right.
\]
where $1_A$ and $0_A$ denote the identity and the zero element of~$A$, respectively. Then $\delta_g\in \lonealgebra$ and $\Vert \delta_g\Vert =1$ for all $g\in G$. Furthermore, $\lonealgebra$ is unital with $\delta_e$ as identity element, where~$e$ denotes the identity element of~$G$. Using \cref{eq:prod}, one finds that
\begin{equation*}
\delta_{gh}=\delta_{g}\cdot \delta_h \quad
\end{equation*}
for all $g,k\in G$. Hence $\delta_g$ is invertible in $\lonealgebra$ for all $g\in G$, and we have $\delta_g^{-1}=\delta_{g^{-1}}$. It is now obvious that the set $\{\,\delta_g : g\in G\,\}$ consists of norm one elements of $\lonealgebra$, and that it is a subgroup of the invertible elements of $\lonealgebra$ that is isomorphic to~$G$.

In the same vein, it follows easily from \cref{eq:prod,eq:involution} that we can view~$A$ isometrically as a closed *-subalgebra of $\lonealgebra$, namely as $\{\,a\delta_e: a\in A\,\}$, where $a\delta_e$ is the element of $\lonealgebra$ that assumes the value $a\in A$ at $e\in G$, and the value $0_A\in A$ elsewhere.

If $\texttt{\bf a}\in \lonealgebra$, then it is easy to see that $\texttt{\bf a}=\sum_{g\in G} (a_g \delta_e)\delta_g$ as an absolutely convergent series in $\lonealgebra$. Hence, if we identify $a_g \delta_e$ and $a_g$, we have $\texttt{\bf a}=\sum_{g\in G} a_g \delta_g$ as an absolutely convergent series in $\lonealgebra$.

Finally, let us note that an elementary computation, using \cref{eq:prod} and the identifications just mentioned, shows that the identity
\begin{equation*}
\delta_g a\delta_g^{-1}=\alpha_g(a)
\end{equation*}
holds in $\lonealgebra$ for all $g\in G$ and $a\in A$.

The following key observation, already used in \cite{de_jeu_el_harti_pinto:2017b}, is now clear.

\begin{lemma}\label{res:lonealgebra_contains_semidirect_product}
Let  $(A,G,\alpha)$ be a \Cstar-dynamical system, where~$G$ is a discrete group and~$A$ a unital \Calgebra\ with unitary group~$U$. Let $H=\set{u\delta_g: u\in U,\, g\in G}$. Then~$H$ and the semidirect product $U\rtimes_\alpha G$ are canonically isomorphic as abstract groups. Moreover, the linear span of~$H$ is dense in $\lonealgebra$.
\end{lemma}

\section{Innerness of derivations of abstract groups}\label{sec:amenability_of_abstract_groups}

This section is concerned with prudent hypotheses ensuring that a bounded derivation of an abstract group~$G$ with values in a dual Banach~$G$-bimodule is inner.

We start by employing Johnson's argument \cite[p.~33]{johnson_COHOMOLOGY_IN_BANACH_ALGEBRAS:1972} under such hypotheses. It seems customary in the literature to first change the~$G$-bimodule structure into one where the left action is trivial, prove that for such~$G$-bimodule structures bounded derivations in dual Banach $G$-modules are inner, and then conclude from this special case that this also holds in the general case. As the proof of \cref{res:core_argument_for_innerness} below shows, this is hardly an actual simplification. It also shows how the map $g\mapsto Dg\cdot g^{-1}$, occurring in the definition of a strongly amenable unital \Calgebra, is quite natural in a more general context.

In the proof, a family of functions occur that will frequently reappear in the sequel. We shall now define these, and subsequently proceed with the innerness of derivations.

\begin{definition}\label{def:phi_function}
Let~$G$ be an abstract group, and let~$E$ be a Banach~$G$-bimodule. Let $D\colon G\to E^\ast$ be a derivation. For $x\in E$, define $\phifunction{x,D}{G}:G\mapsto \CC$ by
\[
\phifunction{x,D}{G}(g)=\angles{x,Dg\cdot g^{-1}}
\]
for $g\in G$.
\end{definition}

\begin{lemma}\label{res:core_argument_for_innerness}
Let~$G$ be an abstract group, and let~$E$ be a Banach~$G$-bimodule where the left $G$-action on~$E$ is unital. 

Let $D\colon G\to E^\ast$ be a bounded derivation. Then $\phifunction{x,D}{G}\in\linf{G}$ for all $x\in E$.

Put
\[
\fspace{D}{G}=\spanlc\set{\set{\phifunction{x,D}{G} : x\in E}\cup\set{\onefunction{G}}}.
\]
Then $\fspace{D}{G}$ is left $G$-invariant.

Suppose that there exists a left $G$-invariant element $\mean_{\fspace{D}{G}}$ of ${\bigpars{\fspace{D}{G}}}^\ast$ with $\angles{\onefunction{G},\mean_{\fspace{D}{G}}}=1$. Then~$D$ is inner.
\end{lemma}

\begin{proof}
It is clear that $\phifunction{x,D}{G}\in\linf{G}$ for all $x\in E$, so we start with the left $G$-invariance of $\fspace{D}{G}$. Let $x\in E$ and $g_0\in G$. Then, for $g\in G$,
\begin{align*}
(\laction{g_0}{G}\phifunction{x,D}{G})(g)&=\phifunction{x,D}{G}(g_0g)\\
&=\angles{x,D(g_0g)\cdot(g_0g)^{-1}}\\
&=\angles{x,Dg_0\cdot g\cdot g^{-1}\cdot g_0^{-1}+g_0\cdot Dg\cdot g^{-1}\cdot g_0^{-1}}\\
&=\angles{x,Dg_0\cdot g_0^{-1}} + \angles{g_0^{-1}\cdot x\cdot g_0,Dg\cdot g^{-1}}.
\end{align*}
Hence
\begin{equation}\label{eq:basic_relation}
\laction{g_0}{G}\phifunction{x,D}{G} = \angles{x,Dg_0\cdot g_0^{-1}}\onefunction{G} + \phifunction{g_0^{-1}\cdot x\cdot g_0,D}{G}
\end{equation}
for all $x\in E$ and $g\in G$.
From this the left $G$-invariance of $\fspace{D}{G}$ is clear.

If an element $\mean_{\fspace{D}{G}}$ of ${\bigpars{\fspace{D}{G}}}^\ast$ as specified exists, then define $x_0^\ast\colon E\to\CC$ by
\[
\angles{x,x_0^*}= \angles{\phifunction{x,D}{G},\mean_{\fspace{D}{G}}}
\]
for $x\in E$. The boundedness of $\mean_{\fspace{D}{G}}$, of~$D$, and of the left $G$-action on $E$ ensure that $x_0^\ast$ is bounded.

Replacing~$g_0$ with~$g$, and~$x$ with $g\cdot x$ in \cref{eq:basic_relation}, we find that
\[
\laction{g}{G}\phifunction{g\cdot x,D}{G} = \angles{x,Dg}\onefunction{G} + \phifunction{x\cdot g,D}{G}
\]
for all $x\in E$ and $g\in G$; it is here that we use that the left $G$-action on~$E$ is unital. Using this, we see that, for all $x\in E$ and $g\in G$,
\begin{align*}
\angles{x,g\cdot x_0^\ast}&=\angles{x\cdot g,x_0^\ast}\\
&=\angles{\phifunction{x\cdot g,D}{G},\mean_{\fspace{D}{G}}}\\
&=\angles{\laction{g}{G}\phifunction{g\cdot x,D}{G} - \angles{x,Dg}\onefunction{G}, \mean_{\fspace{D}{G}}}\\
&=\angles{\phifunction{g\cdot x.D}{G},\mean_{\fspace{D}{G}}} - \angles{x,Dg}\\
&=\angles{g\cdot x, x_0^\ast} - \angles{x,Dg}\\
&=\angles{x, x_0^\ast\cdot g} - \angles{x,Dg}.
\end{align*}
Hence $Dg=x_0^\ast\cdot g - g\cdot x_0^\ast$ for all $g\in G$, so that~$D$ is inner.
\end{proof}

\begin{remark}\label{rem:no_state_like_condition_and_existence_of_a_left_invariant_state}\quad
\begin{enumerate}
\item\label{rem:no_state_like_condition} Note that it is not required in \cref{res:core_argument_for_innerness} that $\norm{\mean_{\fspace{D}{G}}}=1$, i.e.\  we do not impose the state-like condition $\norm{\mean_{\fspace{D}{G}}}=1=\angles{\onefunction{G},\mean_{\fspace{D}{G}}}$.
\item\label{rem:existence_of_a_left_invariant_state} Suppose that there exists a left $G$-invariant element $\mean_{\linf{G}}$ of $\linf{G}^\ast$ such that $\angles{\onefunction{G},\mean_{\linf{G}}}=1$. In that case, \cref{res:core_argument_for_innerness} shows that every bounded derivation of~$G$ with values in an arbitrary dual Banach $G$-bimodule~$E^\ast$ is inner, provided that the left $G$-action on~$E$ is unital. Then \cite[Theorem~11.8.i]{pier_AMENABLE_LOCALLY_COMPACT_GROUPS:1984} implies that the discrete group~$G$ is amenable in the sense of \cite[Definition~4.2]{pier_AMENABLE_LOCALLY_COMPACT_GROUPS:1984} (see also \cite[Definition~3.1]{pier_AMENABLE_LOCALLY_COMPACT_GROUPS:1984}), and consequently \cite[Proposition~3.2]{pier_AMENABLE_LOCALLY_COMPACT_GROUPS:1984} yields that there exists a left $G$-invariant state on $\linf{G}$. Thus the discrete group~$G$ is, in fact, an amenable topological group as this notion was defined in Section~\ref{sec:introduction_and_overview}. However, in order to emphasise that it is the `actual' natural condition, we shall, when applicable, insist on using the seemingly weaker requirement that there exist a left $G$-invariant element $\mean_{\linf{G}}$ of $\linf{G}^\ast$ such that $\angles{\onefunction{G},\mean_{\linf{G}}}=1$.
\end{enumerate}
\end{remark}

\begin{remark}\label{rem:boundedness_of_left_group_action_is sufficient}
Inspection of the proof of \cref{res:core_argument_for_innerness} shows that the hypotheses on the~$G$-actions and~$D$ can be relaxed. It is sufficient that the left and right $G$-actions on~$E$ be by bounded operators, that the left $G$-action on~$E$ be unital, and that $\set{Dg\cdot g^{-1} : g\in G}$ be a norm bounded subset of $E^\ast$.
\end{remark}

The next step is to investigate the innerness of bounded derivations of a group that is generated by two subgroups, one of which is normal. For this result, \cref{res:derivation_is_inner_on_generated_group} below, we need the following preparation.

\begin{lemma}\label{res:core_argument_for_mean}
Let~$G$ be an abstract group that is generated by a subgroup~$H$ and a normal subgroup~$N$.

Let $\fspace{}{N}$ be a left $N$-invariant subspace of $\linf{N}$ containing the constants, and let $\fspace{}{G}$ be a left $G$-invariant subspace of $\linf{G}$ containing the constants and such that $\f\!\!\upharpoonright_N\in\fspace{}{N}$ for all $\f\in\fspace{}{G}$.

Suppose that there exists a left $H$-invariant element $\mean_{\linf{H}}$ of $\linf{H}^\ast$ with $\angles{\onefunction{H},\mean_{\ell^\infty (H)}}=1$, as well as a left $N$-invariant element $\mean_{\fspace{}{N}}$ of ${\fspace{\,\ast}{N}}$ with $\angles{\onefunction{N},\mean_{\fspace{}{N}}}=1$.

Then there exists a left $G$-invariant element $\mean_{\fspace{}{G}}$ of $\fspace{\,\ast}{G}$ with $\angles{\onefunction{G},\mean_{\fspace{}{G}}}=1$.
\end{lemma}

Every $g\in G$ can be written as $g=hn$ with $h\in H$ and $n\in N$, but it is not required that this factorisation be unique, i.e.\ it is sufficient that~$G$ be a quotient of an external semi-direct product of its subgroups~$H$ and~$N$, and not necessarily equal to an internal semi-direct product of these subgroups. When applying \cref{res:core_argument_for_mean} in our principal case of interest (see the proof of \cref{res:derivation_is_inner_on_lonealgebra}) the latter will actually be the case, but it plays no role in the proofs.

\begin{proof}
This is a variation on a construction that is usually carried out in the context of a discrete (see e.g.\ \cite[Proposition~4.5.5]{ceccherini-silberstein_coornaert_CELLULAR_AUTOMATA_AND_GROUPS:2010}) or, more generally (see e.g.\ \cite[Proposition~13.4]{pier_AMENABLE_LOCALLY_COMPACT_GROUPS:1984}), a locally compact  Hausdorff topological group~$G$ having a normal closed subgroup~$N$ such that~$N$ and $G/N$ are both amenable.

For $\f\in\fspace{}{G}$, let $\tilde \f\in\linf{H}$ be defined by
\[
\tilde \f(h)=\angles{(\laction{h}{G} \f)\!\!\upharpoonright_N,\mean_{\fspace{}{N}}},
\]
for $h\in H$, and put
\[
\angles{\f,\mean_{\fspace{}{G}}}=\angles{\tilde \f,\mean_{\linf{H}}}.
\]
Then $\mean_{\fspace{}{G}}\in\fspace{\ast}{G}$, and $\angles{\onefunction{G},\mean_{\fspace{}{G}}}=1$.

It is elementary that $(\laction{h_0}{G} \f)^\sim=\laction{h_0}{H} \tilde \f$ for all $\f\in\fspace{}{G}$ and $h_0\in H$. The $\laction{h_0}{H}$-invariance of $\mean_{\linf{H}}$ then implies that $\mean_{\fspace{}{G}}$ is invariant under $\laction{h}{G}$ for all $h\in H$.

For $\f\in\fspace{}{G}$, $n_0\in N$, and $h_0\in H$ fixed, we have
\begin{align*}
(\laction{n_0}{G} \f)^\sim (h_0)&=\angles{(\laction{h_0}{G} \laction{n_0}{G} \f)\!\!\upharpoonright_N,\mean_{\fspace{}{N}}}\\
&=\angles{(\laction{n_0 h_0}{G} \f)\!\!\upharpoonright_N,\mean_{\fspace{}{N}}}.
\end{align*}
Since~$N$ is normal, there exists $n_0^\prime\in N$ such that $n_0h_0=h_0 n_0^\prime$, and it is elementary that $(\laction{h_0n_0^\prime}{G}\f)\!\!\upharpoonright_N=\laction{n_0^\prime}{N}[(\laction{h_0}{G} \f)\!\!\upharpoonright_N]$. Hence, using the $\laction{n_0^\prime}{N}$-invariance of $\mean_{\fspace{}{N}}$ in the penultimate step, we have
\begin{align*}
(\laction{n_0}{G} \f)^\sim (h_0)&=\angles{(\laction{n_0 h_0}{G} \f)\!\!\upharpoonright_N,\mean_{\fspace{}{N}}}\\
&=\angles{(\laction{h_0 n_0^\prime}{G} \f)\!\!\upharpoonright_N,\mean_{\fspace{}{N}}}\\
&=\angles{\laction{n_0^\prime}{N}[(\laction{h_0}{G} \f)\!\!\upharpoonright_N],\mean_{\fspace{}{N}}}\\
&=\angles{(\laction{h_0}{G} \f)\!\!\upharpoonright_N,\mean_{\fspace{}{N}}}\\
&=\tilde \f(h_0).
\end{align*}
Therefore $(\laction{n_0}{G} \f)^\sim=\tilde f$; it follows that $\mean_{\fspace{}{G}}$ is also $\laction{n}{G}$-invariant for all $n\in N$.

Since~$G$ is generated by~$H$ and~$N$, this concludes the proof.
\end{proof}

We can now combine \cref{res:core_argument_for_innerness} and \cref{res:core_argument_for_mean} and obtain the following result. The prudent formulation of the hypotheses in \cref{res:core_argument_for_innerness,res:core_argument_for_mean} allows us to be likewise prudent in the hypotheses on the group~$N$ in the following result. The functions $\phifunction{x,D}{N}$ figuring in it are the ones as in \cref{def:phi_function}, but then for the restrictions of the actions and of the derivation to~$N$. Obviously, these are simply the restrictions of the $\phifunction{x,D}{G}$ to~$N$.

\begin{proposition}\label{res:derivation_is_inner_on_generated_group}
Let~$G$ be an abstract group that is generated by a subgroup~$H$ and a normal subgroup~$N$, where~$H$ has the property that there exists a left $H$-invariant element $\mean_{\linf{H}}$ of $\linf{H}^\ast$ with $\angles{\onefunction{H},\mean_{\ell^\infty (H)}}=1$.

Let~$E$ be a~Banach $G$-bimodule where the left $G$-action on~$E$ is unital, and let $D\colon G\to E^\ast$ be a bounded derivation. Let the left $N$-invariant subspace $\fspace{D}{N}$ of $\linf{N}$ be defined by
\[
\fspace{D}{N}=\spanlc\set{\set{\phifunction{x,D}{N} : x\in E}\cup\set{\onefunction{N}}},
\]
and suppose that there exists a left $N$-invariant element $\mean_{\fspace{D}{N}}$ of ${\bigpars{\fspace{D}{N}}}^\ast$ with $\angles{\onefunction{N},\mean_{\fspace{D}{N}}}=1$.

Then~$D$ is an inner derivation of~$G$.
\end{proposition}

\begin{proof}
Put
\[
\fspace{D}{G}=\spanlc\set{\set{\phifunction{x,D}{G} : x\in E}\cup\set{\onefunction{G}}}.
\]
The first part of \cref{res:core_argument_for_innerness} shows that $\fspace{D}{N}$ is a left $N$-invariant subspace of $\linf{N}$, and that $\fspace{D}{G}$ is left $G$-invariant subspace of $\linf{G}$. From the comments preceding the proposition, it is trivial that the set of restrictions of the elements of $\fspace{D}{G}$ to~$N$ equals $\fspace{D}{N}$.

We are now in the situation of \cref{res:core_argument_for_mean}, and we conclude that there exists a left $G$-invariant element $m_{\fspace{D}{G}}$ of ${\bigpars{\fspace{D}{G}}}^\ast$ with $\angles{\onefunction{G}, m_{\fspace{D}{G}}}=1$. An appeal to \cref{res:core_argument_for_innerness} shows that~$D$ is an inner derivation of~$G$.
\end{proof}

\begin{remark}\label{rem:boundedness_of_left_group_action_of_generated_group_is sufficient} In view of \cref{rem:boundedness_of_left_group_action_is sufficient} and the above proof, the hypotheses in \cref{res:derivation_is_inner_on_generated_group} on the $G$-actions and~$D$ can be relaxed. It is sufficient that the left and right $G$-actions on~$E$ be by bounded operators, that the left $G$-action on~$E$ be unital, and that $\set{Dg\cdot g^{-1} : g\in G}$ be a norm bounded subset of $E^\ast$.
\end{remark}

\section{Amenability of $\lonealgebra$}\label{sec:amenability_of_lonealgebra}

All that remains to be done is to combine \cref{res:lonealgebra_contains_semidirect_product} and \cref{res:derivation_is_inner_on_generated_group}, and add Paterson's result \cite[Theorem~2]{paterson:1991} later on. This is made possible by the prudence concerning hypotheses in Section~\ref{sec:amenability_of_abstract_groups}.

We recall that the condition on~$G$ occurring in the results in this section is actually equivalent to requiring that there exist a left $G$-invariant state on $\linf{G}$; see part \ref{rem:existence_of_a_left_invariant_state}
of \cref{rem:no_state_like_condition_and_existence_of_a_left_invariant_state}

The following result is the most precise one in this section, because it is concerned with only one bounded derivation of $\lonealgebra$. For the convenience of the reader, the definition of the functions $\phifunction{x,D}{U}$ from \cref{def:phi_function} is included again, as it will be in \cref{res:lonealgebra_is_amenable_general_case}.

\begin{proposition}\label{res:derivation_is_inner_on_lonealgebra}
Let $\dynsys$ be a \Cstar-dynamical system, where~$A$ is a unital \Calgebra\ with unitary group~$U$, and~$G$ is a discrete group such that there exists a left $G$-invariant element $\mean_{\linf{G}}$ of $\linf{G}^\ast$ with $\angles{\onefunction{G},\mean_{\linf{G}}}=1$.

Let~$E$ be a Banach $\lonealgebra$-bimodule, and let $D\colon\lonealgebra\to E^\ast$ be a bounded derivation.
For $x\in E$, define $\phifunction{x,D}{U}:U\mapsto \CC$ by
\[
\phifunction{x,D}{U}(u)=\angles{x,Du\cdot u^{-1}}
\]
for $u\in U$, and put
\[
\fspace{D}{U}=\spanlc\set{\set{\phifunction{x,D}{U} : x\in E}\cup\set{\onefunction{U}}}.
\]

Suppose that there exists a left $U$-invariant element $\mean_{\fspace{D}{U}}$ of ${\bigpars{\fspace{D}{U}}}^\ast$ with $\angles{\onefunction{U},\mean_{\fspace{D}{U}}}=1$.

Then~$D$ is an inner derivation of $\lonealgebra$.
\end{proposition}

\begin{proof}
Assume first that~$E$ is unital as a left $\lonealgebra$-module.

Consider $H\simeq U\rtimes_\alpha G\subset\lonealgebra$. Obviously,~$E$ is a~$H$-bimodule. Since~$H$ is a bounded subset of $\lonealgebra$ (recall that $\norm{h}=1$ for all $h\in H$),~$E$ is, in fact,  a Banach $H$-bimodule. Since~$E$ is unital as a left $\lonealgebra$-module, and the identity elements of $\lonealgebra$ and~$H$ coincide, the left $H$-action on~$E$ is unital.

The derivation~$D$ of $\lonealgebra$ restricts to a derivation of~$H$, again denoted by~$D$. Since~$H$ is a bounded subset of $\lonealgebra$,~$D$ is a bounded derivation of~$U$ with values in $E^\ast$.

We can now apply \cref{res:derivation_is_inner_on_generated_group}, where we replace~$G$ with our present~$H$, $H$ with our present~$G$, and~$N$ with our present~$U$. We conclude from \cref{res:derivation_is_inner_on_generated_group} that~$D$ is inner on our present~$H$. Since by \cref{res:lonealgebra_contains_semidirect_product} the linear span of our present~$H$ is dense in $\lonealgebra$,~$D$ is inner on $\lonealgebra$. This concludes the proof where~$E$ is unital as a left $\lonealgebra$-module.

It remains to cover the case where~$E$ is possibly degenerate as a left $\lonealgebra$-module. The approach for this is more or less standard (cf.\ e.g.\ \cite[Proposition~1.8]{johnson_COHOMOLOGY_IN_BANACH_ALGEBRAS:1972} or \cite[Proposition~0.3]{pier_AMENABLE_BANACH_ALGEBRAS:1988}), but since our criterion applies to one derivation, and is in terms of a specific space of functions on~$U$ that is defined using the \emph{whole} space~$E$, we cannot content ourselves with a reference to a known isomorphism between two cohomology groups, one of which then corresponds to a unital left action on the largest submodule with this property. Hence we include the required steps. Denote the left action of the identity element $\delta_e$ of $\lonealgebra$ on~$E$ by $P$. Then $P$ is a continuous projection with adjoint $P^\ast\colon E^\ast\to E^\ast$. Both are morphisms of $\lonealgebra$-bimodules. The subspace $PE$ is a Banach $\lonealgebra$-bimodule such that the left action of $\lonealgebra$ is unital. In order to describe the dual Banach $\lonealgebra$-bimodule $(PE)^\ast$ of $PE$, define $j\colon P^\ast E^\ast\to (PE)^\ast$ by  $j(x^\ast)=x^\ast\!\!\upharpoonright_{PE}$ for $x^\ast\in P^\ast E^\ast$. Then $j$ is a topological isomorphism of  Banach $\lonealgebra$-bimodules between $P^\ast E^\ast$ and $(PE)^\ast$. Define $D_P\colon\lonealgebra\to(PE)^\ast$ by $j(a)=j(P^\ast Da)$ for $a\in A$. Then $D_P$ is a bounded derivation of $\lonealgebra$ with values in the dual Banach $\lonealgebra$-bimodule $(PE)^\ast$. Moreover, it is easily verified (this is the point) that $\fspace{D_P}{U}=\fspace{D}{U}$. From the result for the left unital case, we know that there exists $\xi^\ast\in (PE)^\ast$ such that $D_P {\texttt{\bf a}}={\texttt{\bf a}}\cdot\xi^\ast-\xi^\ast\cdot {\texttt{\bf a}}$ for all ${\texttt{\bf a}}\in \lonealgebra$. It is then a routine further verification that~$D$ is the inner derivation of $\lonealgebra$ that corresponds to $j^{-1}\xi + (\idmap_{E^\ast}-P^\ast)D\delta_e\in E^\ast$.
\end{proof}

In the unital case, it is now obvious what the natural sufficient condition is for $\lonealgebra$ to be amenable. The following is immediate from \cref{res:derivation_is_inner_on_lonealgebra}.

\begin{theorem}\label{res:lonealgebra_is_amenable_general_case}
Let $\dynsys$ be a \Cstar-dynamical system, where~$A$ is a unital \Calgebra\ with unitary group~$U$, and where~$G$ is a discrete group such that there exists a left $G$-invariant element $\mean_{\linf{G}}$ of $\linf{G}^\ast$ with $\angles{\onefunction{G},\mean_{\linf{G}}}=1$.

For every Banach $\lonealgebra$-bimodule~$E$, every $x\in E$, and every bounded derivation $D\colon\lonealgebra\to E^\ast$, define $\phifunction{x,D}{U}\colon U\mapsto \CC$ by
\[
\phifunction{x,D}{U}(u)=\angles{x,Du\cdot u^{-1}}
\]
for $u \in U$.

Let $\fspace{\mathrm{Der}}{U}$ denote the left $U$-invariant subspace of $\linf{U}$ that is spanned by $\onefunction{U}$ and the functions $\phifunction{x,D}{U}$ corresponding to all Banach $\lonealgebra$-bimodules~$E$, all $x\in E$, and all bounded derivations $D\colon\lonealgebra\to E^\ast$.

Suppose that there exists a left $U$-invariant element $\mean_{\fspace{\mathrm{Der}}{U}}$ of $\bigpars{\fspace{\mathrm{Der}}{U}}^\ast$ with $\angles{\onefunction{U},\mean_{\fspace{}{U}}}=1$. Then $\lonealgebra$ is an amenable Banach algebra.
\end{theorem}

It remains to identify a substantial class of unital \Calgebras\ such that $\mean_{\fspace{\mathrm{Der}}{U}}$ as in \cref{res:lonealgebra_is_amenable_general_case} exists.  As we shall now proceed to show, the strongly amenable \Calgebras\ are such a class.

Let~$A$ be a unital \Calgebra\ with unitary group~$U$. Let $\mathrm{Bil}(A)$ be the space of all bounded bilinear forms on~$A$. Following \cite{paterson:1991}, for $V\in\mathrm{Bil}(A)$, we define $\Delta(V)\colon U\to\CC$ by $\Delta(V)(u)=V(u^{-1},u)$. Let $\fspace{\mathrm{Bil}}{U}=\spanlc\set{\Delta V : V\in \mathrm{Bil}(A)}$; this space is denoted by $B(A)$ in \cite{paterson:1991}. Then $\fspace{\mathrm{Bil}}{U}$ is a subspace of $\linf{U}$ containing the constants (consider $(a_1,a_2)\mapsto a^\ast(a_1a_2)$ for $a^\ast\in A^\ast$ such that $a^\ast(1_A)=1)$, and it is easy to see that it is left and right $U$-invariant; see \cite [p.~557]{paterson:1991}.  It can then be shown (see \cite[Theorem~2]{paterson:1991}) that~$A$ is strongly amenable if and only if there exists a right $U$-invariant element $\mean_{\fspace{\mathrm{Bil}}{U}}$ of $\bigpars{{\fspace{\mathrm{Bil}}{U}}}^\ast$ such that $\mean_{\fspace{\mathrm{Bil}}{U}}(1)=1=\norm{\mean_{\fspace{\mathrm{Bil}}{U}}}$.

Actually, the strong amenability of a unital \Calgebra~$A$ is also equivalent to the existence of a left (which is what we need) $U$-invariant element $\mean_{\fspace{\mathrm{Bil}}{U}}$ of $\bigpars{\fspace{\mathrm{Bil}}{U}}^\ast$ such that $\mean_{\fspace{\mathrm{Bil}}{U}}(1)=1=\norm{\mean_{\fspace{\mathrm{Bil}}{U}}}$. We shall now show this. For $\f\in\linf{U}$, define $\check{\f}\in\linf{U}$ by $\f(u)=\f(u^{-1})$ for $u\in U$. For $V\in\mathrm{Bil}(A)$, define $\check V\in\mathrm{Bil}(A)$ by $\check V(a_1,a_2)=V(a_2,a_1)$ for $a_1,a_2\in A$. Then $(\Delta V)^{\check{}}=\Delta(\check V)$ for $V\in\mathrm{Bil}(A)$. We see from this that the map $\f\mapsto \check \f$ is an isometric linear automorphism of the normed space $\fspace{\mathrm{Bil}}{U}$ and this, in turn, shows how to obtain left $U$-invariant continuous linear functionals from right $U$-invariant ones, and vice versa.

With \cref{res:lonealgebra_is_amenable_general_case} and the above discussion of Paterson's result available, the rest is now easy.

\begin{theorem}\label{res:lonealgebra_is_amenable_if_A_is_strongly_amenable}
Let $\dynsys$ be a \Cstar-dynamical system, where~$A$ is a strongly amenable not necessarily unital \Calgebra, and~$G$ is a discrete group such that there exists a left $G$-invariant element $\mean_{\linf{G}}$ of $\linf{G}^\ast$ with $\angles{\onefunction{G},\mean_{\linf{G}}}=1$.

Then $\lonealgebra$ is an amenable Banach algebra.
\end{theorem}

\begin{proof}
As a consequence of \cite[Lemma~2.2]{de_jeu_el_harti_pinto:2017b} and the definition of strong amenability, we may assume that~$A$ is a strongly amenable unital \Calgebra.

In view of \cref{res:lonealgebra_is_amenable_general_case} and the discussion preceding the present theorem, we need only show that $\fspace{\mathrm{Der}}{U}\subset\fspace{\mathrm{Bil}}{U}$. So let~$E$ be a Banach $\lonealgebra$-bimodule, let $x\in E$, let $D\colon\lonealgebra\to E^\ast$ be a bounded derivation, and consider the associated function $\phifunction{x,D}{U}:U\mapsto \CC$, defined by
\[
\phifunction{x,D}{U}(u)=\angles{x,Du\cdot u^{-1}}
\]
for $u\in U$. Define the bounded bilinear form $V^{x,D}$ on~$A$ by $V^{x,D}(a_1,a_2)=\angles{a_1 \cdot x, Da_2}$ for $a_1, a_2\in A$. Then $\phifunction{x,D}{U}=\Delta (V^{x,D})$.
\end{proof}

It was already noted in Section~\ref{sec:introduction_and_overview} that all Type I (equivalently: all postliminal) \Calgebras\ are strongly amenable. Therefore, \cref{res:lonealgebra_is_amenable_if_A_is_strongly_amenable} implies that $\lonealgebra$ is amenable if~$A$ is a commutative or finite dimensional \Calgebra, which is \cite[Theorem~2.4]{de_jeu_el_harti_pinto:2017b}.

\section{Converses}\label{sec:converses}

We shall now briefly discuss possible converse implications: if $\dynsys$ is a \Cstar-dynamical system where~$G$ is a discrete group, and if $\lonealgebra$ is an amenable Banach algebra, then what can one say about~$G$ or~$A$? The following modest result sums up the results in this vein that the authors are aware of.

\begin{proposition}\label{res:converse}
Let $(G, A,\alpha)$ be a \Cstar-dynamical system where~$G$ is a discrete group. Suppose that $\lonealgebra$ is amenable. Then:
\begin{enumerate}
\item The algebra~$A$ is amenable;
\item If $A=\CC$, then~$G$ is an amenable group;
\item If $A=\cont(X)$ for a compact Hausdorff space~$X$, then the action of~$G$ on~$X$ is amenable.
\end{enumerate}
\end{proposition}

For the definition of an amenable~$G$-action on~$X$ we refer to \cite[Definition~4.3.5]{brown_ozawa_C-STAR-ALGEBRAS_AND_FINITE-DIMENSIONAL_APPROXIMATIONS:2008}.

We note that it is not asserted in the first part that~$A$ is strongly amenable. In view of the logical dependence between \cref{res:lonealgebra_is_amenable_general_case} and \cref{res:lonealgebra_is_amenable_if_A_is_strongly_amenable} this is, perhaps, also not to be expected. After all, since\textemdash in the notation of Section~\ref{sec:amenability_of_lonealgebra}\textemdash there seems to be no reason why the inclusion $\fspace{\mathrm{Der}}{U}\subset\fspace{\mathrm{Bil}}{U}$ could not be proper, one would, perhaps, expect that there are unital \Calgebras\ that are not strongly amenable, but for which \cref{res:lonealgebra_is_amenable_general_case} still applies. On the other hand, in the realm of amenability of groups there are some surprising implications between the existence of invariant functionals on various function spaces, and we cannot entirely exclude the possibility that the conclusion in the first part of \cref{res:converse} can still be strengthened.

\begin{proof}
We prove the first statement. The crossed product \Calgebra\ $A\rtimes_\alpha G$ is the enveloping \Calgebra\ of the involutive Banach algebra $\lonealgebra$. By the very construction of such an enveloping \Calgebra, there is a continuous (even contractive) homomorphism of $\lonealgebra$ into $A\rtimes_\alpha G$ with dense range. Since $\lonealgebra$ is amenable, \cite[Proposition 2.3.1]{runde_LECTURES_ON_AMENABILITY:2002} implies that $A\rtimes_\alpha G$ is amenable. Therefore, the reduced crossed product $A\rtimes_{\alpha,\mathrm{r}} G$, which is a quotient of $A\rtimes_\alpha G$, is also amenable. Since there exists a conditional expectation from~$A\rtimes_{\alpha,\mathrm{r}} G$ onto~$A$ (see e.g.\ \cite[Proposition~4.1.9]{brown_ozawa_C-STAR-ALGEBRAS_AND_FINITE-DIMENSIONAL_APPROXIMATIONS:2008}), this implies that~$A$ is amenable (see e.g.\ \cite[Exercise~2.3.3]{brown_ozawa_C-STAR-ALGEBRAS_AND_FINITE-DIMENSIONAL_APPROXIMATIONS:2008}).

The second statement is, of course, just a special case of Ringrose's converse to Johnson's theorem; see \cite[Theorem~2.5]{johnson_COHOMOLOGY_IN_BANACH_ALGEBRAS:1972}.

For the third statement, one sees as for the first statement that $\cont(X)\rtimes_{\alpha, \mathrm r}G$ is amenable. Then the conclusion follows from \cite[Theorem~4.4.3]{brown_ozawa_C-STAR-ALGEBRAS_AND_FINITE-DIMENSIONAL_APPROXIMATIONS:2008}.
\end{proof}

The proofs of the first and third statements rely heavily on some non-trivial results for \Calgebras. It is also possible to deduce the second statement in this fashion. As before, if $\ell^1(G)$ is amenable, then so is ${\mathrm C}^\ast_{\mathrm r}(G)$. Since~$G$ is discrete, \cite[Theorem~4.2]{lance:1973} then shows that~$G$ is amenable. Compared to this, Ringrose's argument for the amenability of~$G$ is definitely more direct, and it would be desirable to have a more direct argument for the first and third statements as well.

We do not know if, for non-trivial~$A$, the amenability of $\lonealgebra$ implies the amenability of~$G$. An attempt to derive such results via the enveloping \Calgebra\ (as above) does not only seem overly complicated, but may also be the wrong approach, since the passage to the enveloping \Calgebra\ simplifies the structure. For example, $\upL^1(G)$ has a non-selfadjoint closed ideal for every abelian non-compact locally compact Hausdorff topological group (see \cite[Theorem~7.7.1]{rudin_FOURIER_ANALYSIS_ON_GROUPS:1962}), but this is of course no longer true for ${\mathrm C}^\ast(G)$. In a similar vein, and directly related to amenability, we make the following observation in the non-discrete case. It is known (this goes back to \cite{ehrenpreis_mautner:1955}) that $\mathrm{SL}(2,\RR)$ is not amenable. Hence $\upL^1(\mathrm{SL}(2,\RR))$ is not amenable either. However, ${\mathrm C}^\ast(\mathrm{SL}(2,\RR))$ is amenable (see~\cite{connes:1976} for the amenability of every separable connected Lie group, and \cite[p.~228]{paterson:1987} for the general possibly non-separable case). The information that the group is non-amenable is thus lost (or is at least no longer reflected in the amenability of the algebra) when passing from $\upL^1(\mathrm{SL}(2,\RR))$ to ${\mathrm C^\ast}(\mathrm{SL}(2,\RR))$.  Such a lack of information in $A\rtimes_\alpha G$ can also occur for non-trivial~$A$. Indeed, the Cuntz \Calgebra\ $\mathcal{O}_2$ on two generators is amenable, but it was noted independently by Kumjian and Archbold (see  \cite[p.~83]{rordam_stormer_CLASSIFICATION_OF_NUCLEAR_C-STAR-ALGEBRAS:2002} for details) that $\mathcal{O}_2$ can be obtained as a \Calgebra\ crossed product of the continuous functions on the Cantor set $\{0,1\}^{\mathbb{N}}$ by the non-amenable discrete group $\mathbb{Z}_2\ast\mathbb{Z}_3$. Here, again, the information that~$G$ is not amenable is not present in $A\rtimes_\alpha G$, or is at least not reflected in its amenability properties.

It seems, therefore, at least conceivable that $\lonealgebra$ still contains enough information to be able to infer some kind of amenability property of~$G$ (or~$A$) if $\lonealgebra$ itself is amenable, and it also seems that proper $\upL^1$-type arguments will have to be developed in order to investigate such Ringrose-type converses.


\subsection*{Acknowledgements}
We thank Narutaka Ozawa and Jun Tomiyama for helpful comments.
The last author was partially funded by FCT/Portugal through project UID/MAT/04459/2013.




\bibliography{general_bibliography}


\end{document}